\documentclass[11pt]{amsart}
\usepackage{amssymb}
\usepackage{amsmath}
\usepackage{amsthm}
\usepackage{hyperref}
\usepackage{enumerate}
\usepackage{curves}
\usepackage{graphicx}
\usepackage{tikz}
\usepackage{enumitem}
\usepackage{a4wide}
\newcommand{\arxiv}[1]{\href{http://arxiv.org/abs/#1}{arXiv:#1}}
\newcommand{\doi}[1]{\href{http://dx.doi.org/#1}{doi: #1}}

\newtheorem{theorem}{Theorem}[section]

\newtheorem{lemma}[theorem]{Lemma}
\newtheorem{corollary}[theorem]{Corollary}

\newtheorem*{proposition*}{Proposition}

\theoremstyle{definition}
\newtheorem{remark}[theorem]{Remark}
\numberwithin{equation}{section}

\newcommand{\R}{\mathbb{R}}

\newcommand{\C}{\mathbb{C}}

\newcommand{\nn}{\nonumber}
\newcommand{\be}{\begin{equation}}
\newcommand{\ee}{\end{equation}}
\newcommand{\bea}{\begin{eqnarray}}
\newcommand{\eea}{\end{eqnarray}}

\newcommand{\ol}{\overline}
\newcommand{\pa}{\partial}

\newcommand{\id}{\mathbb{I}}
\newcommand{\I}{\mathrm{i}}
\newcommand{\E}{\mathrm{e}}

\newcommand{\re}{\mathop{\mathrm{Re}}}
\newcommand{\im}{\mathop{\mathrm{Im}}}

\DeclareMathOperator{\res}{Res}
\newcommand{\noprint}[1]{}

\makeatletter
\newcommand{\dlmf}[1]{%
\cite[%
  \def\nextitem{\def\nextitem{, }}%
  \@for \el:=#1\do{\nextitem\href{http://dlmf.nist.gov/\el}{(\el)}}%
]{dlmf}%
}
\makeatother

\newcommand{\ga}{\gamma}

\begin{document}
\title[Soliton asymptotics for the KdV shock problem]{Soliton asymptotics for the KdV shock problem of low regularity}

\author[I. Egorova]{Iryna Egorova}
\address{B. Verkin Institute for Low Temperature Physics and Engineering\\ 47, Nauky ave\\ 61103 Kharkiv\\ Ukraine}
\email{\href{mailto:iraegorova@gmail.com}{iraegorova@gmail.com}}

\author[J. Michor]{Johanna Michor}
\address{Faculty of Mathematics\\ University of Vienna\\
Oskar-Morgenstern-Platz 1\\ 1090 Wien\\ Austria}
\email{\href{mailto:Johanna.Michor@univie.ac.at}{Johanna.Michor@univie.ac.at}}
\urladdr{\href{http://www.mat.univie.ac.at/~jmichor/}{http://www.mat.univie.ac.at/\string~jmichor/}}

\author[G. Teschl]{Gerald Teschl}
\address{Faculty of Mathematics\\ University of Vienna\\
Oskar-Morgenstern-Platz 1\\ 1090 Wien\\ Austria\\ and Erwin Schr\"odinger International
Institute for Mathematics and Physics\\ Boltzmanngasse 9\\ 1090 Wien\\ Austria}
\email{\href{mailto:Gerald.Teschl@univie.ac.at}{Gerald.Teschl@univie.ac.at}}
\urladdr{\href{http://www.mat.univie.ac.at/~gerald/}{http://www.mat.univie.ac.at/\string~gerald/}}

\dedicatory{Dedicated to the memory of Sergey Naboko}

\keywords{KdV equation, shock problem, nonlinear steepest descent, low regularity, solitons}
\subjclass[2020]{Primary 37K40, 35Q53; Secondary 37K45, 35Q15}
\thanks{This research was supported by the Austrian Science Fund FWF [grant number P31651].}

\begin{abstract}
We revisit the asymptotic analysis of the KdV shock problem in the soliton region. Our approach is based
on the analysis of the associated Riemann--Hilbert problem and we extend the domain of validity of the asymptotic
formulas while at the same time requiring less decay and smoothness for the initial data.
\end{abstract}

\maketitle

\section{Introduction and main results}

The aim of this note is to revisit the RHP approach (introduced by Deift and Zhou \cite{DZ} extending ideas of Manakov \cite{ma} and Its \cite{its1})
for the study of the long-time asymptotics for solutions of the Korteweg--de Vries (KdV) equation
\be\label{kdv} 
q_t(x,t) - 6 q(x,t)q_x(x,t)+q_{xxx}(x,t)=0
\ee
with step-like initial data $q(x)=q(x,0)$ satisfying the condition
\be\label{ini1}
\lim_{x \to \infty} q(x)= 0, \quad \lim_{x \to -\infty} q(x) = -c^2, \quad c > 0.
\ee
This is known as the KdV shock problem and the solution will split into a decaying dispersive tail, a dispersive shock wave, and a number of solitons. 
Moreover, it was shown by Khruslov \cite{Kh} that at the wave front of the dispersive shock, $x=4c^2t$, solitons will
emerge which are not associated with points of the discrete spectrum.

However, while these principal regions are well understood (\cite{EGKT,EMT1,gp1,gp2}), the regions where the corresponding asymptotics are established do not overlap.
In particular, it is typically a quite delicate task to improve the domain of validity of
these formulas to achieve the aforementioned overlap. In this vein, the aim of the present paper is to refine the Riemann--Hilbert analysis
in the soliton region $x>4c^2t$ to both increase the domain of validity as well as weaken the decay and smoothness requirements for the initial data.
In particular, the degree of decay will appear in the domain of validity. 

More specifically, we assume that the initial data are such that $q(x)\in C^{n_0}(\R)$ and 
\be\label{low}
\int_{\R_+ } |x|^{m_0}\left( |q(x)| + |q(-x)+c^2|\right)dx + \int_\R x^{m_0-1} |q^{(i)}(x)|dx,\quad i=1, \dots, n_0,
\ee
where\be\label{mn}  m_0\geq 4,\quad n_0\geq m_0+3.\ee 
In the following we refer to the Cauchy problem \eqref{kdv}--\eqref{mn} as the {\em KdV shock problem of low regularity}. 
For comparison, the previously available results  for the shock problem \eqref{kdv}--\eqref{mn} from \cite{EGKT}
in the soliton region 
\[
x\geq(4c^2 +\varepsilon)t
\]
were established under the assumption of exponential decay:
\be\label{exp}
\int_{\R_+}  \E^{\rho x}\left( |q(x)| + |q(-x)+c^2|\right)dx, \quad \rho>c.
\ee 
Decaying (non steplike) initial data of low regularity were considered for the KdV equation in \cite{GT} (with $c=0$, $m_0=6$, $n_0=3$)
and for the mKdV equation in \cite{lenells1}. Both results are obtained outside a small sector containing the transition region $\frac x t\sim 0$, that is, for $x>\varepsilon t$. 

 In this connection, two interesting questions arise: Is it possible  to expand the boundary of the soliton region for the KdV shock problem (and thus narrowing the boundaries of the transition region) using the RHP approach? And is it even possible to achieve this under the low regularity assumptions \eqref{low}?
  
Using the classical inverse scattering transform the multisoliton asymptotics were recently derived in \cite{EMT1} in the expanded soliton region   
\be \label{solist}x>4c^2 t +\frac{m_0 -3/2 -\varepsilon}{2c}\log t, \quad m_0\ge 3.
\ee 
Namely, assume that the discrete spectrum of the scattering problem associated with \eqref{kdv}--\eqref{mn} is given by
$-\kappa_N^2<\dots<-\kappa_1^2$, and that the corresponding norming constants of the right eigenfunctions are given by 
$\gamma_j$, $j=1,\dots,N$. Then for $t\to \infty$ uniformly in the domain \eqref{solist} the solution to \eqref{kdv}--\eqref{mn} can be represented as
\be \label{asmain}
q(x,t)  = q^{sol}(x,t)+ O\bigg(\frac{1}{t^{m_0-\frac{3}{2}-\varepsilon}}\bigg),
\ee
where
\be\label{qsol}
q^{sol}(x,t)=-\sum_{j=1}^N\frac{2\kappa_j^2}{\cosh^2\big(\kappa_j x - 4\kappa_j^3 t -\frac{1}{2}\log\frac{\gamma_j^2}{2\kappa_j}-\sum_{i=j+1}^N\log\frac{\kappa_j - \kappa_i}{\kappa_i + \kappa_j}\big)}.
\ee
In the present paper we use the RHP approach to show the following result:

\begin{theorem}\label{mainint}
Assume that the initial datum $q(x)$ satisfies \eqref{low}--\eqref{mn}, does not have a resonance at the edge of the continuous spectrum $-c^2$,
and has a nonempty discrete spectrum $-\kappa_N^2<\dots<-\kappa_1^2$. Assume that $x\to\infty$, $t\to\infty$ such that 
\be\label{delD}
(x,t)\in \mathcal D:= \left\{x\geq 4c^2 t + \frac{\beta}{c}\log t, \quad t\gg 1, \quad \beta\geq 0\right\}.
\ee
 Then in the domain $\mathcal D$ we have
 \be
 q(x,t)  = q^{sol}(x,t)+O\bigg(\frac{1}{t^{\nu}}\bigg), \quad \nu=\min\{m_0-3, \beta+1\}\geq 1.
 \ee
\end{theorem}

The proof is based, among other things, on a new matrix solution of the underlying model problem and will be given in Subsections~\ref{sub2.2}, \ref{sub2.3} and Section~\ref{sec:mpaa}.

Our restrictions \eqref{mn} on the regularity assumptions \eqref{low} on the initial datum were made such that one can guarantee 
the existence of a unique classical solution $q(x,t)$ of \eqref{kdv} remaining within the realm of classical scattering theory, i.e.\ such
that \eqref{first} below holds for all times, as established in \cite{EMT1}. However, \eqref{low}, \eqref{mn} is
not necessary for such a solution to exist and if existence of a classical solution satisfying \eqref{first} is known by other means
(see e.g.\ \cite{GR} for results in this direction), the minimal estimates used to prove Theorem~\ref{mainint} imply the following:

\begin{corollary}[Largest possible class]\label{inter}
Assume that for the initial nonresonant data $q(x)$ satisfying \eqref{low} with $m_0=4$ and $n_0=5$, a unique classical solution of \eqref{kdv}--\eqref{low} exists and satisfies 
\be\label{first}
\int_0^{+\infty}|x|(|q(x,t)| + |q(-x,t)+c^2|)dx<\infty, \qquad\forall t\in\R.
\ee
Then  the following asymptotic is valid for $t\to\infty$ uniformly in the domain $ \  x\geq 4c^2 t$:
 \[q(x,t)=q^{sol}(x,t) +O(t^{-1}).\]
 \end{corollary}

\section{From the initial RHP to the pre-model RHP}

\subsection{Statement of the initial RH problem}

Let $q(x)$ be as in Theorem~\ref{mainint} and let $q(x,t)$ be the solution of \eqref{kdv}--\eqref{mn}. Condition \eqref{mn} implies (cf.\ \cite{EMT1}) that this solution exists, is unique and satisfies \eqref{first}.
Let $\phi(k,x,t)$ be the right Jost solution of the associated Schr\"{o}dinger equation
\[
L(t)y=-\frac{d^2}{dx^2} y + q(x,t) y=k^2 y,
\]
satisfying 
\be\label{limsi}
\lim_{x \to  +\infty} \E^{- \I kx} \phi(k,x,t) =1,
\ee
and let $\phi_1(k,x,t)$ be the corresponding Jost solution associated with the left background,
\be\label{lims}
 \lim_{x \to  -\infty} \E^{ \I k_1 x} \phi_1(k,x,t) =1, \qquad k_1:=\sqrt{k^2 +c^2}.
\ee
Here $k_1>0$ for $k\in[0,\I c)_r$. The subscript "$r$" in the last notation indicates the right side of the cut along the interval $[0,\I c]$. Note that the function $\phi(k,x,t)$ is a holomorphic function of $k$ in $\C^+:=\{ k \in \C \colon \im k > 0\}$ and continuous up to the real axis. It  is real-valued for $k\in [0,\I c]$, and does not have a discontinuity on this interval. The function $\phi_1(k,x,t)$ is holomorphic in the domain $\C^+\setminus (0, \I c]$ and continuous up to the boundary. On the different sides of $[0,\I c]$ it takes complex conjugated values.
Denote the Wronskian of the Jost solutions by 
\[W(k)= \phi_1(k,x,0)\phi^\prime(k,x,0) -\phi_1^\prime(k,x,0)\phi(k,x,0),\] 
where $f^\prime=\frac{\pa}{\pa x} f$. 
The conditions of Theorem~\ref{mainint} exclude a possible resonance at the point $\I c$, that is, we assume the condition 
\[
W(\I c)\neq 0.
\]
On $[0,\I c]$ introduce the function
\be\label{defchi1}
\chi(k):=\frac{4 k \left[k_1\right]_r}{| W(k)|^2}.
\ee
One can verify that $\chi(k)=\I|\chi(k)|$, and $\chi(\I c)=\chi(0)=0$.

Let $R(k)$ be the right reflection coefficient of the initial data $q(x)$ and let
\[
\gamma_j:=\|\phi(\I\kappa_j,\cdot,0)\|^{-2}_{L^2(\R)}
\]
be the  right norming constants for $j=1, \dots, N$. The set
\be\label{scat5}
\{ R(k), k\in \R; \quad |\chi(k)|, k\in[0, \I c];\quad \I\kappa_j, \gamma_j, j=1, \dots,N\},
\ee
constitutes the minimal set of scattering data to uniquely reconstruct the solution of the initial value problem \eqref{kdv}--\eqref{mn} (cf.\ \cite{EGLT, EMT1}). 

The Jost  solutions \eqref{lims} and \eqref{limsi} are connected by the  scattering relation
\[
T(k,t) \phi_1(k,x,t) =  \ol{\phi(k,x,t)} +
R(k,t) \phi(k,x,t),  \quad k \in \R,
\]
where $T(k,t)$ and $R(k,t)$  are the right  transmission and reflection coefficients. We will use the notation $T(k) = T(k,0)$ and
$R(k)=R(k,0)$. 

We define a vector-valued function $m(k,x,t) = (m_1(k,x,t), m_2(k,x,t))$, meromorphic with respect to the spectral parameter $k \in \C \setminus(\R \cup [-\I c, \I c])$ for fixed $x,t$, as follows:
\be\label{defm}
m(k,x,t)= \left\{\begin{array}{cl}
\begin{pmatrix} T(k,t) \phi_1(k,x,t) \E^{\I kx},  & \phi(k,x,t) \E^{-\I kx} \end{pmatrix},
& k\in \C^+\setminus (0,\I c], \\
m(-k,x,t)\sigma_1,
& k\in\C^-\setminus [-\I c, 0),
\end{array}\right.
\ee
where $\sigma_1=(\begin{smallmatrix} 0&1\\1&0\end{smallmatrix})$ is the first Pauli matrix. The vector function $m(k,x,t)$ has at most simple poles at the points $\pm \I \kappa_j$. 
For $k\to \infty$, the following asymptotic formula holds
\be\label{main1}
q(x,t)=\lim_{k\to\infty} 2 k^2 \big(m_1(k,x,t)m_2(k,x,t) -1\big),
\ee 
which we will use to extract our asymptotics.

Let $\varepsilon>0$ and $\delta>0$ be two arbitrary small parameters. We divide the domain 
\[\mathcal D:= \left\{(x,t): x\geq 4c^2 t + \frac{\beta}{c}\log t, \quad t\gg T_0\gg 1, \quad \beta\geq 0\right\}
\] 
into a union of the following regions
\[\aligned
D_N &=\{(x,t)\in\mathcal D:\  x\geq (4\kappa_N^2 +\varepsilon)t\}; \\ 
D_N^{sol}&= \{(x,t)\in\mathcal D: (4\kappa_N^2 -\varepsilon)t\leq 
 x\leq (4\kappa_N^2 +\varepsilon)t\},\\
D_j&=\{(x,t)\in\mathcal D: (4\kappa_{j}^2 -\varepsilon)t \geq x\geq (4\kappa_{j-1}^2 +\varepsilon)t\},\\
D_j^{sol}&= \{ (x,t)\in\mathcal D: (4\kappa_j^2 -\varepsilon)t\leq 
 x\leq (4\kappa_j^2 +\varepsilon)t\}, \quad j=1, \dots, N-1,\\
 D_0&=\{ (x,t): 4c^2 t +\tfrac{\beta}{c}\log t\leq x\leq (4\kappa_1^2 -\varepsilon)t\}.
 \endaligned
 \]
 Denote the small nonintersecting circles around the points of the discrete spectrum by
  \[\mathbb D_j:=\{k: |k-\I\kappa_j|<\delta\}, \quad
\mathbb T_j:= \partial \mathbb D_j =\{k:|k-\I\kappa_j|=\delta\},\quad j=1, \dots ,N,\]
 with counterclockwise oriented boundaries (compare Fig.~\ref{fig:1}).
Let $\mathbb T_j^*=\{ k: -k\in \mathbb T_j\}$ be small circumferences around the points $-\I \kappa_j$,
again with counterclockwise orientation.

Introduce the functions
\be\label{Blaschke}P_j(k):=\prod_{l=j}^N \frac{k+\I\kappa_j}{k-\I\kappa_j},\quad k\in\C^+, \quad j=1, \dots, N;\quad  P_{N+1}(k)=1,\ee
and the matrices
\be\label{vspom}A_j(k)=\begin{pmatrix} 1 & 0 \\
-\frac{\I \gamma_j^2 \E^{t\Phi(\I \kappa_j)} }{k- \I \kappa_j} & 1 \end{pmatrix},\quad B_j(k)=\begin{pmatrix} 1& -\frac{k-\I\kappa_j}{\I\gamma_j^2\E^{2t\Phi(\I\kappa_j)}}\\0&1\end{pmatrix},\quad  j=1, \dots, N,\ee
where $A_j(k)=A_j(k,x,t)$, $B_j(k)=B_j(k,x,t)$. The {\em phase function $\Phi(k)=\Phi(k,x,t)$} is defined by  
\[%\label{phasef}
\Phi(k)= 4 \I k^3+\I k \frac {x}{t}, \quad k\in\C.
\]
In the domain $(k, x, t)\in \C^+\times D_j$,  $j=1, \dots, N-1$, we redefine $m(k)$ given by \eqref{defm} as
\be \label{defmj} m(k,j)= m(k,x,t,j)=\begin{cases}\begin{array}{ll} 
m(k) A_l(k)[P_{j+1}(k)]^{-\sigma_3}, &\quad  k\in \mathbb D_l, \qquad 1\leq l\leq j,\\
m(k)B_l(k)[P_{j+1}(k)]^{-\sigma_3}, &\quad  k\in \mathbb D_l, \qquad N\geq l>j,\\
m(k)[P_{j+1}(k)]^{-\sigma_3},&\quad  k\in (\C^+\setminus (0, \I c])\setminus \cup_{l=1}^N \ol{\mathbb D_l}.  
\end{array}\end{cases}\ee
For $(k,x,t)\in \C^+\times D_N$ we set
\be\label{defmN}m(k,x,t,N)=\begin{cases}\begin{array}{lll} m(k) A_l(k), & k\in \mathbb D_l, &l\leq N,\\
m(k),& k\in (\C^+\setminus (0, \I c])\setminus \cup_{l=1}^N \ol{\mathbb D_l}. & \end{array}\end{cases}\ee
In the domain $(k, x, t)\in \C^+\times D_j^{sol}$, $j=1, \dots, N$, we set
\be \label{defmjsol} m^{sol}(k,j)= m^{sol}(k,x,t,j)=\begin{cases}\begin{array}{ll} 
m(k) A_l(k)[P_{j+1}(k)]^{-\sigma_3}, & \quad k\in \mathbb D_l, \qquad 1\leq l< j,\\
m(k)B_l(k)[P_{j+1}(k)]^{-\sigma_3}, & \quad k\in \mathbb D_l, \qquad N \geq  l>j,\\
m(k)[P_{j+1}(k)]^{-\sigma_3},& \quad k\in (\C^+\setminus (0, \I c])\setminus \cup_{l=1}^N \ol{\mathbb D_l}. 
\end{array}\end{cases}\ee
Redefine $m(k,j)=m(-k,j)\sigma_1$ and $m^{sol}(k,j)=m^{sol}(-k,j)\sigma_1$ for $k\in \C^-.$

We next introduce the jump contour 
\be\label{defcont}\Sigma= \R_+\cup\R_-\cup [\I c,0]\cup [ -\I c,0]\cup_{l=1}^N (\mathbb T_l \cup \mathbb T_l^*)
\ee 
as depicted in Fig.~\ref{fig:1} with the following orientation: left-to-right on $\R_+$, right-to-left on $\R_-=\R_+^*$, 
top-down on $[\I c, 0]$,  bottom-top on $[-\I c, 0]=[\I c, 0]^*$, and counterclockwise on $\mathbb T_j$ and $\mathbb T_j^*$. 
By $I^*$ we refer to the contour consisting of the points $-k: k\in I$ with the following orientation: if $k$ moves 
in the positive direction of $I$, then $-k$ moves in the positive direction of $I^*$.

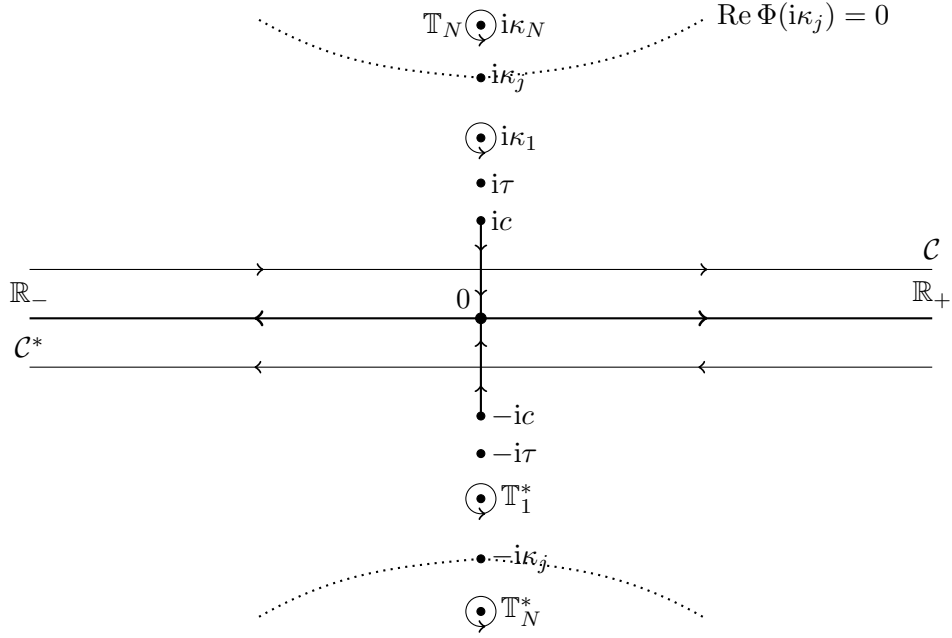
\begin{figure}[ht] 
\begin{tikzpicture} 

%coordinate system and lines
\draw[thick] (-6,0) -- (6,0) node[above] {$\R_+$};
\node[above] at (-6,0) {$\R_-$};
\draw[] (-6,0.65) -- (6,0.65) node[above] {$\mathcal C$};
\draw[] (-6,-0.65) -- (6,-0.65);
\node[above] at (-6,-0.65) {$\mathcal C^*$};
\draw[thick] (0,-1.3) -- (0,1.3);
\filldraw (0,0) circle (2pt) node[above left]{${0}$};

%arrows
\draw [->, very thick] (3,0) -- (3.01,0);
\draw [->, very thick] (-3,0) -- (-3.01,0);
\draw [->, thick] (3,0.65) -- (3.01,0.65);
\draw [->, thick] (-2.9,0.65) -- (-2.89,0.65);
\draw [->, thick] (2.9,-0.65) -- (2.89,-0.65);
\draw [->, thick] (-3,-0.65) -- (-3.01,-0.65);

\draw [->, thick] (0.05,2.2) -- (0.06,2.2);
\draw [->, thick] (0.05,3.7) -- (0.06,3.7);
\draw [->, thick] (0.05,-2.6) -- (0.06,-2.6);
\draw [->, thick] (0.05,-4.1) -- (0.06,-4.1);

\draw [->, thick] (0,0.3) -- (0,0.29);
\draw [->, thick] (0,0.9) -- (0,0.89);
\draw [->, thick] (0,-0.3) -- (0,-0.29);
\draw [->, thick] (0,-0.9) -- (0,-0.89);

%label
\filldraw (0,1.3) circle (1.5pt) node[right]{$\I c$};
\filldraw (0,-1.3) circle (1.5pt) node[right]{$- \I c$};
\filldraw (0,1.8) circle (1.5pt) node[right]{$\I \tau$};
\filldraw (0,-1.8) circle (1.5pt) node[right]{$-\I \tau$};
\filldraw (0,2.4) circle (1.5pt) node[right]{$\ \I \kappa_1$};
\filldraw (0,-2.4) circle (1.5pt) node[right]{$\ \mathbb T_1^*$};
\filldraw (0,3.2) circle (1.5pt) node[right]{$\I \kappa_j$};
\filldraw (0,-3.2) circle (1.5pt) node[right]{$-\I \kappa_j$};
\filldraw (0,3.9) circle (1.5pt) node[right]{$\ \I \kappa_N$};
\filldraw (0,-3.9) circle (1.5pt) node[right]{$\ \mathbb T_N^*$};

%contour
\draw[out=5, in=210, thick, dotted] (0,3.2) to (3,4) node[right]{$\re \Phi(\I\kappa_j)=0$};
\draw[out=175, in=330, thick, dotted] (0,3.2) to (-3,4);
\draw[out=355, in=150, thick, dotted] (0,-3.2) to (3,-4);
\draw[out=185, in=30, thick, dotted] (0,-3.2) to (-3,-4);

%circles
\draw (0,3.9) circle (0.2cm);
\draw (0,-3.9) circle (0.2cm);
\node at (-0.5,3.9) {$\mathbb T_N$}; 
\draw (0,2.4) circle (0.2cm);
\draw (0,-2.4) circle (0.2cm);

\end{tikzpicture}
\caption{Part of the jump contour $\Sigma$.} \label{fig:1}
\end{figure}

We observe that $m(k,j)$ is a piecewise holomorphic vector function with jumps on $\Sigma$ and $m^{sol}(k,j)$ is a piecewise meromorphic function with simple  poles at $\I \kappa_j$ and $-\I\kappa_j$ and with the same jumps as $m(k,j)$ except at  $\mathbb T_j$ and $\mathbb T_j^*$, where it does not have jumps.
Note that \be\label{1}m_1(k,x,t,j)m_2(k,x,t,j)=m_1(k,x,t)m_2(k,x,t),\quad k\to\infty, \quad (x,t)\in D_j,
\ee 
and
\be\label{2}m_1^{sol}(k,x,t,j)m_2^{sol}(k,x,t,j)=m_1(k,x,t)m_2(k,x,t),\quad k\to\infty, \quad (x,t)\in D_j^{sol}.\ee 
\begin{theorem}\label{thm:vecrhp} 
Let \eqref{scat5} be the right scattering data of the initial datum $q(x)$.
Assume that $x$ and $t$ are arbitrary large fixed  values such that $(x,t)\in D_j$ (resp.\ $(x,t)\in D_j^{sol}$).
 Then the vector function $m(k,j)=m(k,x,t,j)$ (respectively  $m^{sol}(k,j)=m^{sol}(k,x,t,j)$) defined in \eqref{defmj}, \eqref{defmN} (resp.\ \eqref{defmjsol})
is the unique solution of the following vector Riemann--Hilbert problem:

\noindent Find a vector-valued function $m(k,j)$, holomorphic (resp.\ $m^{sol}(k,j)$, meromorphic) away from $\Sigma$ (resp.\ away from $\Sigma\setminus (\mathbb T_j\cup \mathbb T_j^*)$),  satisfying:
\begin{enumerate}
\item The jump condition:

 $m_+(k,j)=m_-(k,j) v(k,j)$ (resp.\ $m_+^{sol}(k,j)=m_-^{sol}(k,j) v^{sol}(k,j)$), where
\[ %\label{eq:jumpcond}
v(k,j)=\left\{\begin{array}{lr}\begin{pmatrix}
1-|R(k)|^2 & - \ol{R(k)}P_{j+1}^{2}(k) \E^{-2t\Phi(k)} \\
P_{j+1}^{-2}(k)R(k) \E^{2t\Phi(k)} & 1
\end{pmatrix},& k\in\R_+,\\[4mm]
\begin{pmatrix}
1 & 0 \\
\chi(k)P^{-2}_{j+1}(k) \E^{2t\Phi(k)} & 1
\end{pmatrix},& k\in [\I c, 0],\\[4mm]
 A_l(k)[P_{j+1}(k)]^{-\sigma_3} & k\in \mathbb T_l, \ \ l\leq j;\\[2mm]
 B_l(k)[P_{j+1}(k)]^{-\sigma_3} & k\in \mathbb T_l, \ \ l>j;\\[2mm]
  \sigma_1 v(-k)\sigma_1,& k\in \R_-\cup[ -\I c,0]\cup_{l=1}^N \mathbb T_l^*;
\end{array}\right.
\]
(resp.\ 
$ %\label{vsol}
v^{sol}(k,j)=v(k,j),\ k\in \Sigma\setminus(\mathbb T_j\cup \mathbb T_j^*);\quad 
v^{sol}(k,j)=\id,\ k\in \mathbb T_j\cup \mathbb T_j^*.)$
  \item
 the pole conditions:
\[
\aligned %\label{eq:polecond}
\res_{\I\kappa_j} m^{sol}(k,j) &= \lim_{k\to\I\kappa_j} m^{sol}(k,j)
\begin{pmatrix} 0 & 0\\ \I \ga_j^2 P_{j+1}^{-2}(\I \kappa_j) \E^{2t\Phi(\I \kappa_j)}  & 0 \end{pmatrix},\\
\res_{-\I\kappa_j} m^{sol}(k,j) &= \lim_{k\to -\I\kappa_j} m^{sol}(k,j)
\begin{pmatrix} 0 & - \I \ga_j^2 P_{j+1}^{2}(\I \kappa_j) \E^{2t\Phi(\I \kappa_j)} \\ 0 & 0 \end{pmatrix},
\endaligned
\]
\item
the symmetry conditions:
\[ %\label{eq:symcond}
m(-k,j) = m(k,j) \sigma_1,\quad  \ k\in \C\setminus\Sigma,
\]
(resp.\
$
%\label{symsol}\ 
m^{sol}(-k,j) = m^{sol}(k,j) \sigma_1,\ \ k\in \C\setminus(\Sigma\setminus (\mathbb T_j\cup\mathbb T_j^*)).)
$
\item
 the normalization condition
\[ %\label{normco}
\lim_{\kappa\to\infty} m(\I\kappa,j) = \lim_{\kappa\to\infty} m^{sol}(\I\kappa,j) = (1\ \ 1).
\]
\item The function $m(k,j)$ (resp.\ $m^{sol}(k,j)$)  has continuous limits as $k$ approaches  $\Sigma$. \end{enumerate}
\end{theorem}
\begin{remark} The results listed in this theorem  are slight modifications of the results obtained in \cite{EGKT, EPT}, and we omit the proof.
\end{remark} 

\subsection{Properties of the scattering data and their analytic continuations}\label{sub2.2}

First assume that the initial datum $q(x)$ is smooth and satisfies \eqref{exp}. Then the reflection coefficient $R(k)$ 
has an analytic continuation to $\mathcal O_\rho^- \cup \mathcal O_\rho^+$, where
$\mathcal O_\rho^\pm=\{k: \pm\re k>0,  0<\im k<\rho\}$.

In contrast to the case of fast decaying initial datum, in the steplike case the analytic continuation of $R(k)$ has a jump along $[\I c, 0]\cap \{0<\im k< \rho\}$, given by (cf.\ \cite[Lemma 3.2]{EGKT})
\be\label{anal1} R_-(k)-R_+(k) + \chi(k)=0,\quad  k\in [\I c, 0]\cap \{0<\im k< \rho\}.\ee
On the other hand, if the initial datum satisfies \eqref{low},  which is the case we consider, $R(k)$ and $\chi(k)$ are  $m_0-1$  times continuously differentiable except for the node point $k=0$, and 
the following formula is valid (\cite[Section~3]{EMT1})
\[%\label{anal2}
 (\I)^{l+1} R^{(l)}(+0) + (-\I)^{l+1} \ol{R^{(l)}(+0)}=\lim_{h\to +0}\frac{d^l}{dh^l}\chi(\I h),
 \] 
where\[ \quad R^{(l)}(+0)=\lim_{k\to +0} \frac{d^l}{dk^l}R(k), \quad l=0,1, \dots, m_0-1.
\]
Respectively,
\be\label{anal3}
\sum_{l=0}^{m_0-1} \frac{R^{(l)}(+0)}{l!} k^l- \sum_{l=0}^{m_0-1} \frac{\ol{R^{(l)}(+0)}}{l!} (-k)^l= \sum_{l=0}^{m_0-1} \frac{\chi^{(l)}(0)}{l!} k^l, \quad \mbox{for $k=\I h$}.
\ee
Thus, this formula agrees with \eqref{anal1} for the case \eqref{low}. In fact, the decomposition of $\chi$ at $0$ has 
only odd degrees of $k$ and 
\be\label{sootn}R^{(l)}(+0)=(-1)^{l} \ol{R^{(l)}(+0)}.
\ee

Set 
$\tau=\frac{c+\kappa_1}{2}.
$ In the domain $\mathcal O^+$, where \be\label{mathcalO}\mathcal O^\pm:=\{k: \pm\re k>0,\quad 0<\im k<\frac c 2\},\ee introduce a rational function
\[ %\label{defg}
q(k)=\sum_{s=n_0+1}^{n_0+m_0}\frac{a_s}{(k-\I\tau)^s},
\] where the coefficients $a_s\in\C$ can be uniquely defined to satisfy
\be\label{anal4}\sum_{l=0}^{m_0-1} \frac{R^{(l)}(0)}{l!} k^l= \sum_{l=0}^{m_0-1} \frac{q^{(l)}(+0)}{l!} k^l.\ee
An elementary analysis of the algebraic system of equations for $a_s$ together with \eqref{sootn} implies
\[\aligned & a_{n_0+ 2s}\in \R, \  \ a_{n_0+2s+1}\in \I \R, \quad \mbox{for}\ n_0\ \mbox{odd},\\
& a_{n_0+ 2s}\in \I\R, \ \  a_{n_0+2s+1}\in  \R, \quad \mbox{for}\ n_0\ \mbox{even},\endaligned\]
and therefore
\[%\label{appro} 
\ol{q(-k)}=q(k),\quad k\in\R.\]
The same property is valid for the reflection coefficient, $R(-k)=\ol{R(k)}$, $k\in \R$.

We showed that the function
\[\aligned
p(k)&=(-1)^{n_0+1}\left( \frac{a_{n_0+1}}{(k+\I\tau)^{n_0+1}}-\frac{a_{n_0+2}}{(k+\I\tau)^{n_0+2}} 
+ \frac{a_{n_0+3}}{(k+\I\tau)^{n_0+3}} \right. \\
&\qquad \left. + \dots + (-1)^{m_0-1} \frac{a_{n_0+m_0}}{(k+\I\tau)^{n_0+m_0}}\right)
\endaligned
\]
satisfies \be\label{imp77}p(k)=\ol{q(k)}=q(-k),\quad k\in \R,\ee and is the Tailor decomposition for $R(k)$ as $k\to -0$. 
On the other hand, it is analytic in $\mathcal O^-$ (cf.\ \eqref{mathcalO}).
Therefore we have
\[%\label{approx4}
R(k)- q(k)=O(k^{m_0-1}),\quad \mbox{as $k\to +0$};\quad R(k)- p(k)=O(k^{m_0-1}),\quad  \mbox{as $k\to -0$};
\]
and from \eqref{anal3} and \eqref{anal4} it follows that
\[%\label{imp11}
p(\I h-0) - q(\I h +0)+\chi(\I h)= O(h^{m_0-1}), \quad \mbox{as $h\to +0$}.
\]
By \cite[Theorem 4.1]{EGLT},
\[
\frac{d^l}{d k^l} R(k)=O\left(\frac{1}{k^{n_0+1}}\right), \quad k\to \pm\infty,\quad l=0,1, \dots, m_0-1.
\]
Therefore,
\[%\label{imp66}
\aligned &\frac{d^l}{d k^l} (R(k)-q(k))=O\left(\frac{1}{k^{n_0+1}}\right), \quad k\to +\infty,\quad l=0,1, \dots, m_0-1.\\
&\frac{d^l}{d k^l} (R(k)-p(k))=O\left(\frac{1}{k^{n_0+1}}\right), \quad k\to -\infty,\quad l=0,1, \dots, m_0-1.
\endaligned\]
Introduce the function 
\[%\label{defF}
\mathcal R(k)=\begin{cases} \begin{array}{ll} R(k)-q(k), &k\geq 0,\\ 
R(k)-p(k), & k<0,\end{array}\end{cases}\]
then we proved the following
\begin{lemma}\label{lem1} (1) The function  $\mathcal R(k)$ has the following properties
\begin{itemize}
\item $\mathcal R\in C^{m_0-1}(\R)$;
\item $\frac{d^l}{d k^l} \mathcal R(k)=O(\frac{1}{k^{n_0+1}})\quad k\to \pm\infty,\quad l=0,1,\dots,m_0-1;$
\item $\frac{d^l}{d k^l} \mathcal R(0)=0,\quad l=0,1,\dots,m_0-1;$
\end{itemize}
(2) For the function \be\label{chivn}f(k):= \chi(k) +p(k-0) - q(k+0), \quad k\in [\tfrac {\I c}{2}, 0],\ee the estimate is valid:
\be\label{imp99}f(k)=O(k^{m_0-1}), \quad k\to 0.\ee
\end{lemma}

Set now 
\[%\label{mathG} 
\mathcal G(k)=\frac{(k-\I\tau)^6}{k^2}\mathcal R(k).\]
By Lemma~\ref{lem1} we have
\be\label{mathcalG} \mathcal G\in  C^{m_0-3}(\R);\quad  \frac{d^l}{d k^l} \mathcal G(0)=0; \quad \frac{d^l}{d k^l} \mathcal G(k)=O\bigg(\frac{1}{k^{n_0-3}}\bigg), \quad k\to\infty, \quad l=0,\dots,m_0-3.
\ee 
Evidently, $\mathcal G\in L^1(\R)\cap L^{\infty}(\R)$. Let 
\[\hat {\mathcal G}(x)=\frac{1}{2\pi}\int_\R \mathcal G(k)\E^{\I k x}dk,\]
be the Fourier transform for $\mathcal G(k)$. 
From Lemma \ref{lem1} and \eqref{mathcalG} we get  \begin{corollary}\label{cor1}
The following properties are valid:
\begin{itemize}
\item $\mathcal G(-k)=\ol{\mathcal G(k)};$
\item $\hat {\mathcal G}(x)\in C^{\,n_0-4}(\R);\quad $ $\hat {\mathcal G}(x)\in\R$;
\item $x^{m_0-3}\hat {\mathcal G}(x)\in L^1(\R)$.
\end{itemize}
\end{corollary}

In terms of $\hat {\mathcal G}$ the function $\mathcal R$ can be represented as a sum,
\[\aligned \mathcal R(k)&=\frac{k^2}{(k-\I\tau)^6}\int_{-c^2 t}^{c^2 t} \hat {\mathcal G}(x)\E^{\I k x}dx +\frac{k^2}{(k-\I\tau)^6}\left(\int_{-\infty}^{-c^2 t}+\int_{c^2 t}^\infty\right)\hat{\mathcal  G}(x)\E^{\I k x}dx \\
&=:\mathcal R_a(k,t) + \mathcal R_r(k,t).
\endaligned\]
The function $\mathcal R_a(k,t)$ can be continued analytically in the strip $0<\im k< \frac c 2$.
Moreover, from Corollary~\ref{cor1} it follows that
\be\label{imp88} \mathcal R_a(-k,t)=\ol{\mathcal R_a(k,t)},\quad \mathcal R_r(-k,t)=\ol{\mathcal R_r(k,t)},\quad k\in \R.\ee
\begin{lemma}\label{lem4}
The following estimates hold:
\begin{enumerate} 
\item $\mathcal R_r(\cdot,t)\in L^1(\R)\cap L^{\infty}(\R)$;
\item $|\mathcal R_r(k,t)|\leq \frac{C}{t^{m_0-3}}\frac{k^2}{k^6 +1}$, $k\in\R$;
\item $|\mathcal R_a(k,t)|\leq C\frac{|k|^2}{|k|^{6}+ 1}\E^{c^2|\im k|t}.$
\end{enumerate}
\end{lemma}

\subsection{Estimates for the jump matrices}\label{sub2.3}

Recall that the jump matrix $v(k,j)=v^{sol}(k,j)$ satisfies for $k\in \R$ the symmetry property
\be\label{ext}v(k,j)=\sigma_1 v(-k,j)\sigma_1.\ee On $\R_+$  we factorize this matrix in the following way:
\be\label{see}v(k,j)=\sigma_1 [V(-k,j)]^{-1}[W(-k,j)]^{-1}\sigma_1 Y(k,j) W(k,j)V(k,j),\quad k\in \R_+,
\ee
where
\be\label{ygr}Y(k,j)=\begin{pmatrix}
1-|\mathcal R_r(k,t)|^2 & - \mathcal R_r(-k,t)P_{j+1}^{2}(k) \E^{-2t\Phi(k)} \\
P_{j+1}^{-2}(k)\mathcal R_r(k,t) \E^{2t\Phi(k)} & 1
\end{pmatrix},
\ee
\be\label{nui}
V(k,j)=\begin{pmatrix} 1&0\\ q(k)P_{j+1}^{-2}(k)\E^{2t\Phi(k)}&1\end{pmatrix},\quad W(k,j)=\begin{pmatrix} 1&0\\ \mathcal R_a(k,t)P_{j+1}^{-2}(k)\E^{2t\Phi(k)}&1\end{pmatrix}.
\ee
Respectively, due to \eqref{imp88} and \eqref{imp77} one has
\[\sigma_1 V(-k,j)\sigma_1=\begin{pmatrix} 1& - \ol{q(k)}P_{j+1}^{2}(k)\E^{-2t\Phi(k)}\\0&1\end{pmatrix},
\ \sigma_1W(-k,j)\sigma_1=\begin{pmatrix} 1& - \ol{\mathcal R_a(k)}P_{j+1}^{2}(k)\E^{-2t\Phi(k)}\\0&1\end{pmatrix}.\]
Extend now the factorization \eqref{see} to $\R_-=\R_+^*$ by \eqref{ext}. 
The matrices $V(k,j)$ and $W(k,j)$ admit analytic continuations to the domains $\mathcal O^\pm$ defined by \eqref{mathcalO}.
In $\mathcal O^+$ they read as in \eqref{nui} while in $\mathcal O^-$ we have
\[
V(k,j)=\begin{pmatrix} 1&0\\ p(k)P_{j+1}^{-2}(k)\E^{2t\Phi(k)}&1\end{pmatrix},\quad W(k,j)=\begin{pmatrix} 1&0\\ \mathcal R_a(k,t)P_{j+1}^{-2}(k)\E^{2t\Phi(k)}&1\end{pmatrix}.
\]
 Introduce symmetric domains in the lower half plane, 
 \[[\mathcal O^\pm]^*=\{k: -k\in \mathcal O^\pm\}.
 \]
Redefine $m(k,j)$ (resp., $m^{sol}(k,j)$) in $\mathcal O^+$ and in $[\mathcal O^-]^*$ by 
\[\aligned
\hat m(k,j)&:=m(k,j)[V(k,j)]^{-1}[W(k,j)]^{-1}, \qquad k\in \mathcal O^+, \\
 \hat m(k,j)&:=m(k,j)\sigma_1[V(-k,j)]^{-1}[W(-k,j)]^{-1}\sigma_1,\qquad k\in [\mathcal O^-]^*,
 \endaligned\]
and extend this redefinition to $\mathcal O^-\cup[\mathcal O^+]^*$ by the symmetry
$\hat m(-k,j)=\hat m(k,j)\sigma_1$. 
In the remaining region $\C\setminus\big(\ol{\mathcal O^-\cup[\mathcal O^+]^*\cup\mathcal O^+\cup[\mathcal O^-]^*}\big)$ we keep $\hat m(k,j)=m(k,j)$.

Introduce additional contours  as in Fig.~\ref{fig:1}, 
\[%\label{hatsig}
\mathcal C=\{k\in\C: \im k=\frac c 2\}, \quad\mbox{and} \ \  
\mathcal C^* =\{k\in\C: \im k=- \frac c 2\},
\]
oriented left-to-right and right-to-left.  Let us split the new jump contour $\hat \Sigma$ into symmetric parts with 
respect to $k\mapsto -k$ and denote 
\[\hat\Sigma=\Sigma_{\mathcal C}\cup\Sigma_{\mathcal C}^*=\Sigma \cup\mathcal C\cup \mathcal C^*,\]
where $\Sigma$ is defined in \eqref{defcont}. Then {\it the  vector function $\hat m(k,j)$ (resp.\ $\hat m^{sol}(k,j)$) is the unique piecewise holomorphic (resp.\ meromorphic with two simple poles at $\pm \I\kappa_j$) solution of the  jump problem  \[\hat m_+(k,j)=\hat m_-(k,j)\hat v(k,j)\quad  (\mbox{resp.}\  \hat m_+^{sol}(k,j)=\hat m_-^{sol}(k,j)\hat  v^{sol}(k,j)),\]
 where
\be \label{eqj}
\hat v(k,j)=\left\{\begin{array}{ll}
Y(k,j),& k\in\R_+,\\[2mm]
 \left(W(k,j)V(k,j)\right)^{-1}, & k\in \mathcal C,\\[2mm]
[W_-(k,j)V_-(k,j)]^{-1}v(k,j)V_+(k,j)W_+(k,j), & k\in [ \frac {\I c}{2},0],\\[2mm] 
v(k,j),& k\in [\I c, \frac {\I c}{2}]\cup_{l=1}^N \mathbb T_l,\\[2mm]
  \sigma_1 \hat v(-k,j)\sigma_1,& k\in \Sigma^*_{\mathcal C}.
\end{array}\right.
\ee

\[(\mbox{resp.}\ \hat v^{sol}(k,j)=\hat v(k,j),\ \ k\in \Sigma_{\mathcal C}\cup\Sigma_{\mathcal C}^*\setminus(\mathbb T_j\cup \mathbb T_j^*);\quad 
\hat v^{sol}(k,j)=\id,\ \ k\in \mathbb T_j\cup \mathbb T_j^*.)\]

The pole conditions for $\hat m^{sol}(k,j)$, the symmetry and normalization conditions are the same as for $m(k,j)$ (resp.\ 
$m^{sol}(k,j)$).
}
\begin{lemma}\label{estvhat} For $x\geq 4c^2 t+\frac{\beta}{c}\log t$, we have
\be\label{estun} \|k^s\,\left(\hat  v(k,j)-\id\right)\|_{L^1(\hat \Sigma)\cap L^\infty (\hat\Sigma)}=O(t^{-\nu}), \quad \nu=\min\{m_0-3,\ \beta+1\}, \quad s=0,1,2. \ee
\end{lemma}
\begin{proof} Estimate \eqref{estun} is trivial for the matrices \eqref{vspom} on the circles $\mathbb T_l$ and $\mathbb T_l^*$. It is also evident for $Y(k,j)$ due to Lemma \ref{lem4}, item 2. Thus, we already have
\be\label{estun1} \|k^{s}\,\left(\hat  v(k,j)-\id\right)\|_{L^1(\tilde\Sigma)\cap L^\infty (\tilde \Sigma)}=O(t^{-m_0+3}), \quad \tilde\Sigma=\R\cup_{l=1}^N (\mathbb T_l\cup\mathbb T_l^*). \ee
Now let $k=\frac {\I c}{2} +y\in \mathcal C \cap \{k: \re k\geq 0\}$. 
The only nonzero (off diagonal) element of the jump matrix $W^{-1}V^{-1}-\id$ can be estimated as 
\[\aligned |W_{21}(k,j)+V_{21}(k,j)|&\leq |P_{j+1}(k) |\cdot|g(k)+\mathcal R_a(k,t)|\E^{2t \re \Phi(k)}\leq C \frac{|k|^2}{|k-\I\tau|^{6}}\E^{\frac{c^3}{2} t+ 2t \re \Phi(k)}\\
&\leq C(m_0, j)\frac{1}{y^4 + 1}\E^{t(- \frac{5}{2} c^3 -12 c y^2)}, \quad y>0, \quad x\geq 4c^2 t.\endaligned\]
 On the remaining parts of $\mathcal C$ and $\mathcal C^*$ the estimates are literally the same. As a result we get
\[%\label{khorest}
\|k^{s}\,\left(\hat  v(k,j)-\id\right)\|_{L^1(\mathcal C\cup \mathcal C^*)\cap L^\infty (\mathcal C\cup \mathcal C^*)}\leq \E^{-Ct}, \quad C>0, \quad x\geq 4c^2 t.\]
It remains to estimate $\hat v(k,j)-\id=\hat v^{sol}_j(k,j)-\id$ on the contour $[\I c, 0]$. The only nonzero  element of this
matrix is \be\label{chino}\hat v_{21}(k,j)=P_{j+1}^{-2}(k)\E^{2t\Phi(k)}
\begin{cases}\aligned \chi(k), &\quad k\in [\I c,\tfrac {\I c}{2}],\\ f(k), &\quad k\in [\tfrac {\I c}{2},0],\endaligned\end{cases}\ee
where the function $f(k)$ continuous on $[\tfrac {\I c}{2},0]$ is defined in \eqref{chivn} and satisfies \eqref{imp99}.
Note that the analytic continuation of $\mathcal R_a(k,t)$ does not have a jump on $[\tfrac {\I c}{2}, 0]$. This was taken into account  to get \eqref{chino}. 

Consider first the contour $[\I c,\tfrac {\I c}{2}]$. Recall that we study the nonresonant case and therefore
\[\chi(k)=C(k-\I c)(1+o(1)), \quad k\to \I c, \quad C\neq 0,\] 
see \eqref{lims}, \eqref{defchi1}. For $x\geq 4 c^2 t + \zeta$, $\zeta\geq 0$, an elementary estimate 
holds for $k=\I\kappa$, $\tfrac c 2 \leq \kappa\leq c$,
\[ t\Phi(\I\kappa, x, t)=4(\kappa^2 - c^2)\kappa t -\kappa\zeta<(\kappa -c)\frac{3 c^2}{4} t - \frac c2 \zeta\leq 0.\]
Therefore,
\be\label{hatv1}|\hat v_{21}(\I\kappa,j)|\leq C (c-\kappa)\E^{-C(c-\kappa) -c\zeta}.\ee
For all regions under consideration except of $D_0$ we have $\zeta\geq (4\kappa_1^2 - 4 c^2 - \varepsilon)t$, and therefore 
\[ %\label{hatv2}
|\hat v_{21}(k,j)|+|\hat v^{sol}_{21}(k,j)|\leq C\E^{-\delta t},\quad k\in [\I c, \tfrac {\I c}{2}],\quad (x,t)\in D_j\cup D_j^{sol},
\] where $ C=C(j)>0$, $\delta=\delta(j)>0$, $j=1,\dots, N.$
In the domain $D_0$, which depends on $\beta$, consider first the case $\beta=0$, that is,  $\zeta=0$. 
 We observe that the function on the r.h.s.\ in \eqref{hatv1} can be estimated from above via  $\max_{y\in\R_+}u(y,t)$, where
$u(y,t)=y\E^{-C y t}$. Since $\frac{\pa u}{\pa y}=0$ for $y=\frac{1}{Ct}$ and $\max_{y\in\R_+}u(y,t)=\frac{1}{C t}$, we conclude that 
$|\hat v_{21}(k,0)|=O(\frac{1}{t})$ uniformly with respect to $k$ and $x$ in the domain under consideration.
If $\beta>0$, that is, $D_0=\{(x,t): x\geq 4c^2 t + \frac{\beta}{c} \log t\}$, then
\[%\label{estcr}
|\hat v_{21}(k,0)|=O(t^{-\beta-1}), \quad k\in [\I c, \tfrac {\I c}{2}],\quad (x,t)\in D_0.
\]
It remains to estimate $\hat v_{21}(k,j)=\hat v^{sol}_{21}(k,j)$ on the interval $[\tfrac {\I c}{2}, 0]$.
Taking into account \eqref{chino} and \eqref{imp99} we conclude that
\[|\hat v_{21}(\I h,j)|\leq C (j) h^{m_0 - 1}\E^{8h^3t - 8c^2ht -2\frac{\beta}{c} h \log t}\leq C  \max_{h\in[0,\tfrac c 2]}h^{m_0-1}\E^{-4c^2t h},\  j=0, \dots, N.
\]
Since 
\[\max_{y\in\R_+}y^{m_0-1}\E^{-Cyt}\leq \frac{C(m_0)}{t^{m_0-1}},\]
then
\be\label{fin}|\hat v_{21}(k,j)| +|\hat v^{sol}_{21}(k,j)|\leq \frac{C(m_0,j)}{t^{m_0-1}},\quad k\in [\tfrac {\I c}{2}, 0].\ee
Collecting the estimates \eqref{estun1}--\eqref{fin} together and taking into account the symmetry property \[\hat v(k,j)=\sigma_1\hat v(-k,j)\sigma_1,\quad k\in\hat \Sigma,\] we get \eqref{estun}. 
\end{proof}
Hence we can reformulate our pre-model RHP as follows

\begin{theorem} \label{pre-mod} Assume that $(x,t)\in D_j$  (resp.\ $(x,t)\in D_j^{sol}$). Then 
\begin{itemize} \item the  vector function $\hat m(k,j,x,t)$ (resp.\ $\hat m^{sol}(k,j,x,t)$) is the unique piecewise holomorphic (resp.\ meromorphic) solution of the  jump problem  \[\hat m_+(k,j)=\hat m_-(k,j)(\id +\mathcal W(k,j))\quad  (\mbox{resp.}\  \hat m_+^{sol}(k,j)=\hat m_-^{sol}(k,j)(\id +\mathcal W(k,j))), \quad k\in \hat \Sigma,\]
 where the jump matrix $\mathcal W(k,j)$ satisfies the symmetry condition
 \be\label{symW}
 \mathcal W(k,j)=\sigma_1\mathcal W(-k,j)\sigma_1, \quad k\in \hat \Sigma,
 \ee
 and the following estimate:
 \be\label{nusi}\|k^s\,\mathcal W(k,j)\|_{L^1(\hat \Sigma)\cap L^\infty (\hat\Sigma)}=O(t^{-\nu}), \quad \nu=\min\{m_0-3,\ \beta+1\}, \quad s=0,1,2. \ee
 In a vicinity of the point $k=0$,
 \be\label{appr} \mathcal W(k,j)=O(k^2), \quad k\in\hat \Sigma,\quad k\to 0.\ee
 \item
 Both functions $\hat m(k,j)$  and $\hat m^{sol}(k,j)$ satisfy the same symmetry conditions
 \be \label{eqsym}
\hat m(-k,j) =\hat  m(k,j) \sigma_1,\quad  \ k\in \C\setminus \hat\Sigma,
\ee
and  normalization conditions
\be\label{normcond0}
\lim_{\kappa\to\infty}\hat m(\I\kappa,j) = \lim_{\kappa\to\infty} \hat m^{sol}(\I\kappa,j) = (1\ \ 1).
\ee
\item
The function $\hat m^{sol}(k,j)$ has simple poles at $\pm\I \kappa_j$ and satisfies
the pole conditions
\be
\aligned \label{eqpo}
\res_{\I\kappa_j} \hat m^{sol}(k,j) &= \lim_{k\to\I\kappa_j} \hat m^{sol}(k,j)
\begin{pmatrix} 0 & 0\\ \I \ga_j^2(x,t)   & 0 \end{pmatrix},\\
\res_{-\I\kappa_j} \hat m^{sol}(k,j) &= \lim_{k\to -\I\kappa_j} \hat m^{sol}(k,j)
\begin{pmatrix} 0 & - \I \ga_j^2 (x,t) \\ 0 & 0 \end{pmatrix},
\endaligned
\ee
where
\be\label{gamo}\gamma_j^2(x,t)=\gamma_j^2\E^{8\kappa_j^3 t-2\kappa_j x}\prod_{l=j+1}^N \left(\frac{\kappa_l - \kappa_j}{\kappa_l + \kappa_j}\right)^2, \quad (x,t)\in D_j^{sol}.\ee
\item In a vicinity of $k=0$,
\be\label{esti} \hat m^{sol}_{1}(k,j)=\hat m^{sol}_{2}(k,j) +O(k^2), \quad k\to 0.\ee
\end{itemize}
\end{theorem}
\begin{proof} All propositions of Theorem \ref{pre-mod} are proven except of \eqref{appr} and \eqref{esti}. Estimate \eqref{appr} is straightforward from \eqref{eqj}, \eqref{ygr}, \eqref{chino}, \eqref{imp99}, \eqref{mn} (for at least $m_0\geq 4$) and Lemma \ref{lem4}.

For \eqref{esti}, we have the following: (1) the jump matrix $\hat v^{sol}(k,j)$ in a vicinity of point $k=0$ satisfies $\hat v^{sol}(k,j)=\id +O(k^2)$; (2) the vector function $m^{sol}(k,j)$ is bounded there and has continuous limits when approaching the contour;  (3) it  satisfies \eqref{eqsym}. Then
\[\hat m^{sol}_{1,+}(k,j)=\hat m^{sol}_{2,+}(k,j)+O(k^2)=\hat m^{sol}_{1,-}(k,j)+O(k^2)=\hat m^{sol}_{2,-}(k,j)+O(k^2), \quad k\to 0,\]
which in turn implies \eqref{esti}.
\end{proof}

\section{Solution of the model problem and final asymptotic analysis}
\label{sec:mpaa}

The model problem for $\hat m(k,j)$ is trivial: To find a vector function holomorphic in $\C$ and satisfying the symmetry and normalization condition. Its unique solution is the constant vector $m^{mod}(k,j)=(1,1)$, which is the same for all domains $D_j$, $j=0,1, \dots, N$.

The second model problem, namely to find a vector function $m^{mod,sol}(k,j)=\mathcal S(k,j)$ meromorphic in $\C$ and satisfying \eqref{eqsym}--\eqref{eqpo}, was solved in \cite{GT}. The (unique) vector solution $\mathcal S(k,j)$ is given by  (cf.\ \cite[Lemma 2.6 and Theorem 4.4]{GT})
\[ %\label{solj}
\aligned \mathcal S(k,j)&=\left(\mathcal S_1(k,j),\, \mathcal S_2(k,j)\right),\quad \mathcal S_2(k,j)=\mathcal S_1(-k,j),\\
 \mathcal S_1(k,j)&=\frac{1}{1+(2\kappa_j)^{-1}\gamma_j^2(x,t)}\left(
1+\frac{k+\I\kappa_j}{k-\I\kappa_j}(2\kappa_j)^{-1}\gamma_j^2(x,t)\right),\endaligned
\]
where $\gamma_j^2(x,t)$ is defined by \eqref{gamo}.
Verification of the pole, symmetry and normalization conditions is straightforward. 

To apply the standard final asymptotic analysis  for the vector RH problems, we need to construct a matrix solution $M(k,j)$ to the model problem in $D_j^{sol}$, which satisfies the additional symmetry $M(-k,j)=\sigma_1 M(k,j)\sigma_1$.  In $D_j$ it will evidently be the identity matrix. For $D_j^{sol}$ we cannot expect that a bounded invertible symmetric matrix solution exists. Indeed, we observe that 
\[\mathcal S_1(0,j)=\mathcal S_2(0,j),\] and
\[\mathcal S_1(0,j)=\mathcal S_2(0,j)=0, \quad \mbox{for}\quad  1-(2\kappa_j)^{-1}\gamma_j^2(x,t)=0.\]
The set of pairs $(x,t)$ satisfying this condition is a line in $D_j^{sol}$ containing arbitrary large $x$ and $t$. According to \cite{EPT}, for such $(x,t)$  a bounded symmetric invertible matrix model solution does not exist. When admitting poles for $M(k,j)$, one has to ensure that the error vector $\hat m(k,j)M^{-1}(k,j)$ has only removable singularities. 

Let us first construct an antisymmetric vector solution for the model problem in $D_j^{sol}$ with a simple pole at $k=0$. 
We look for a solution of the form
\[ \aligned \mathcal V(k,x,t,j)=\mathcal V(k,j)&=\left(-1+\frac{\rho_j}{k} + \frac{\mu_j}{k-\I\kappa_j},\ 1+\frac{\rho_j}{k} + \frac{\mu_j}{k+\I\kappa_j}\right), \quad \im k>0,\\
 \mathcal V(-k,j)&=-\mathcal V(k,j),\endaligned\] where the constants $\rho_j=\rho_j(x,t)$ and $\mu_j=\mu_j(x,t)$ are chosen to satisfy the pole conditions and the condition $\mathcal V_2(\I\kappa_j,j)=\mathcal S_2(\I\kappa_j,j)$.
Then \[\mathcal V_2(\I\kappa_j,j)=1 +\frac{\rho_j}{\I\kappa_j} +\frac{\mu_j}{2\I\kappa_j}=\frac{1}{1+(2\kappa_j)^{-1}\gamma^2_j(x,t)}=\mathcal S_2(\I\kappa_j,j)\]
and
\[ \aligned \mu_j&=\res_{\I\kappa_j} \mathcal V_1(\I\kappa_j,j)=\I\gamma^2_j(x,t)\mathcal V_2(\I\kappa_j,j)\\ &=\res_{\I\kappa_j} \mathcal S_1(\I\kappa_j,j)=\frac{\I\gamma^2_j(x,t)}{1+(2\kappa_j)^{-1}\gamma^2_j(x,t)}.\endaligned\]
We get
\be\label{munu} 
\rho_j(x,t)=-\frac{\I \gamma_j^2(x,t)}{1+(2\kappa_j)^{-1}\gamma^2_j(x,t)},\quad \mu_j(x,t)= \frac{\I \gamma^2_j(x,t)}{1+(2\kappa_j)^{-1}\gamma^2_j(x,t)}, \quad \rho_j=-\mu_j;\ee
\[%\label{mathV}
 \mathcal V(k,j)=\left(-1 +\frac{\I\kappa_j\, \mu_j(x,t)}{k(k-\I\kappa_j)},\,1 - \frac{\I\kappa_j\, \mu_j(x,t)}{k(k+\I\kappa_j)}\right).
 \]
In terms of $\mu_j$ the solution $\mathcal S(k,j)=m^{mod, sol}(k,j)$ has the representation
\be\label{mathS}\mathcal S(k,j)=\left( 1 +\frac{\mu_j(x,t)}{k-\I\kappa_j},\ 1 -\frac{\mu_j(x,t)}{k+\I\kappa_j}\right).\ee
For $(x,t)\in D_j^{sol}$  introduce the matrix 
\be\label{matrix} M(k,j)=M(k,j,x,t)=\frac{1}{2}\begin{pmatrix} \mathcal S_1(k,j)-\mathcal V_1(k,j) & \mathcal S_2(k,j)-\mathcal V_2(k,j)\\
\mathcal S_1(k,j)+\mathcal V_1(k,j)& \mathcal S_2(k,j)+\mathcal V_2(k,j)\end{pmatrix}.\ee
From \eqref{gamo}--\eqref{matrix} it follows that
\be\label{matrxsol} M(k,j)=\begin{pmatrix} 1+\frac{\mu_j(x,t)}{2 k} & -\frac{\mu_j(x,t)}{2k}\frac{k-\I\kappa_j}{k+\I \kappa_j}\\ \frac{\mu_j(x,t)}{2k}\frac{k+\I\kappa_j}{k-\I \kappa_j}& 1-\frac{\mu_j(x,t)}{2 k}\end{pmatrix}, \quad \mu_j(x,t)=\frac{\I\gamma_j^2(x,t)}{1+(2\kappa_j)^{-1}\gamma_j^2(x,t)}.\ee
Its evident properties are listed in the following
\begin{lemma}\label{lemS} The matrix $M(k,j)=M(k,j,x,t)$ given by \eqref{gamo}, \eqref{matrxsol} for $(x,t)\in D_j^{sol}$ satisfies
\begin{itemize} \item symmetry: $M(-k,j)=\sigma_1 M(k,j)\sigma_1$;
\item normalization: $M(k,j)\to \id$ as $k\to\infty$;
\item the matrix $M(k,j)$ is  meromorphic in $\C$ with simple poles at $\pm \I \kappa_j$ and $k=0$;
\item $\det M(k,j)=1$;
\item $m^{mod,sol}(k,j)= (1,\,1)M(k,j)$;
\end{itemize}
\end{lemma}

\begin{theorem}
Let $\hat m^{sol}(k,j)$ be the solution of RHP \eqref{symW}--\eqref{gamo} and let
\[%\label{mer}
m^{err}(k,j):= \hat m^{sol}(k,j) M^{-1}(k,j), \quad (x,t)\in D_j^{sol}.
\] 
Then
\begin{enumerate}
\item the vector function $m^{err}(k,j)$ does not have singularities at $\pm\I\kappa_j$;
\item  $m^{err}(k,j)$ does not have a singularity at $k=0$;
\item the following estimate is valid uniformly for $(x,t)\in D_j^{sol}$:
\[
\|k^s\,M(k,j)\mathcal W(k,j) [M(k,j)]^{-1}\|_{L^1(\hat \Sigma)\cap L^\infty (\hat\Sigma)}=O(t^{-\nu}),\]
where
\[ \nu=\min\{m_0-3,\ \beta+1\}, \ s=0,1,2.\]
\end{enumerate}
\end{theorem}

\begin{proof}
Consider first the point $ \I \kappa_j$. By definition,
\be\label{merf}
\aligned
&m^{err}(k,j)= \left(\hat m^{sol}_1(k,j), \ \hat m^{sol}_2(k,j)\right)\begin{pmatrix} 
1-\frac{\mu_j(x,t)}{2 k}  & \frac{\mu_j(x,t)}{2k}\frac{k-\I\kappa_j}{k+\I \kappa_j}\\ -\frac{\mu_j(x,t)}{2k}\frac{k+\I\kappa_j}{k-\I \kappa_j}& 1+\frac{\mu_j(x,t)}{2 k}\end{pmatrix},\\ \nonumber
&\lim_{k\to\I \kappa_j} (k-\I\kappa_j)\hat m^{sol}_1(k,j)=\I\gamma^2_j(x,t)\hat m^{sol}_2(\I\kappa_j,j).
\endaligned
\ee
Since $[M(\cdot,j)]^{-1}_{12}=O(k-\I\kappa_j)$, then $m^{err}_2(\I\kappa_j,j)$ is well defined. For the first element of this vector we have using \eqref{munu}
\[\lim_{k\to\I \kappa_j} (k-\I\kappa_j)m^{err}_1(\I\kappa_j,j)=\I\gamma^2_j(x,t)\hat m^{sol}_2(\I\kappa_j,j)\left(1 -\frac{\mu_j(x,t)}{2\I\kappa_j}\right)-
\mu_j(x,t)\hat m^{sol}_2(\I\kappa_j,j)=0.\]
The arguments for the point $-\I\kappa_j$ are the same due to symmetry,
$m^{err}(k,j)=m^{err}(-k,j)\sigma_1$.
To prove item 2, note that by \eqref{esti} and \eqref{merf},
\[\aligned m^{err}(k,j)&= (K,\,K) M^{-1}(k,j) +O(k)\\ &= K\left(1 -\frac{\mu_j}{2k}(1 + \frac{k+\I\kappa_j}{k-\I \kappa_j}),\quad 1 +\frac{\mu_j}{2k}(1 + \frac{k-\I\kappa_j}{k+\I \kappa_j})\right)+ O(k)\\
&=K\left(1-\frac{\mu_j(x,t)}{k-\I\kappa_j},\ 1+\frac{\mu_j(x,t)}{k+\I\kappa_j}\right) + O(k), \quad K=
m^{err}_1(+0+\I0,j).
\endaligned\]
To prove item 3, we observe that $\mu_j(x,t)$ defined by  \eqref{munu}, \eqref{matrxsol}, \eqref{gamo}
is uniformly bounded in $D_j^{sol}$. Therefore,
$M(k,j)$ and $M^{-1}(k,j)$ admit an estimate from above of the form $1+\frac{C}{|k|}$, where $C$ does not depend on $(x,t)$ and $k\in \hat\Sigma$. So outside a small vicinity of $k=0$ the estimate in item 3  is fulfilled because of \eqref{nusi}. Near $k=0$ we apply \eqref{appr}.
\end{proof}

The rest of the final asymptotic analysis is a trivial modification of the standard "small norm" arguments
for symmetric vector RH problems.
Indeed, let $\hat m(k,j)$ and $\hat m^{sol}(k,j)$ be as in Theorem \ref{pre-mod} and let $M(k,j)$ be defined by \eqref{matrix} in $D_j^{sol}$ and by the identity matrix in $D_j$. Set
\[m^{er}(k)=\begin{cases}\begin{array}{ll} \hat m(k,j), & (x,t)\in D_j,\\
\hat m^{sol}(k,j)[M(k,j)]^{-1}, & (x,t)\in D_j^{sol},\end{array}\end{cases}\quad j=0, \dots, N,
\]
and define
\[ \mathcal W^{er}(k)= M(k,j)\mathcal W(k,j) [M(k,j)]^{-1}, \quad j=0, \dots, N.
\]
Then $m^{er}(k)=m^{er}(-k)\sigma_1$ is the unique piecewise holomorphic solution of the jump problem
\[
%\label{merr}
m^{er}_+(k)=m^{er}_-(k)(\id + \mathcal W^{er}(k)), \quad k\in\hat\Sigma,
\]
where 
\[
\mathcal W^{er}(k)=\sigma_1 \mathcal W^{er}(-k) \sigma_1, \quad k\in\hat\Sigma,
\]
and
\be\label{nusi1}\|k^s\,\mathcal W^{er}(k)\|_{L^1(\hat \Sigma)\cap L^\infty (\hat\Sigma)}=O(t^{-\nu}), \quad \nu=\min\{m_0-3,\ \beta+1\}\geq 1, \quad s=0,1,2. \ee
This vector function is continuous up to the boundaries and satisfies the normalization condition
\[%\label{normc}
\lim_{\kappa\to\infty} m^{er}(\I\kappa) =  (1, \ 1).
\]

Let $\mathfrak C$ denote the Cauchy operator associated with $\hat\Sigma$,
\[(\mathfrak C h)(k)=\frac{1}{2\pi\I}\int_{\hat\Sigma}h(s)\frac{ds}{s-k}, \qquad k\in\C\setminus\hat\Sigma,
\]
where $h= (h_1,  h_2)\in L^2(\hat\Sigma)$. 
Let  $\mathfrak C_+ f$ and $\mathfrak C_- f$ be its non-tangential limit values from the left and right sides of 
$\hat\Sigma$, respectively.

As usual, we introduce the operator $\mathfrak C_{W}:L^2(\hat\Sigma)\cup L^\infty(\hat\Sigma)\to
L^2(\hat\Sigma)$ by $\mathfrak C_{W} f=\mathfrak C_-(f \mathcal W^{er})$, where $\mathcal W^{er}$ is our error matrix. 
Then
\[
\|\mathfrak C_{W}\|_{L^2(\hat\Sigma)\to L^2(\hat\Sigma)}\leq C\|\mathcal W^{er}\|_{L^\infty(\hat\Sigma)}\leq O(t^{-\nu})
\] 
as well as
\be\label{6w}
\|(\id - \mathfrak C_{W^{er}})^{-1}\|_{L^2(\hat\Sigma)\to L^2(\hat\Sigma)}\leq \frac{1}{1-O(t^{-\nu})}
\ee
for sufficiently large $t$. Consequently, for $t\gg 1$, we may define a vector function
\[
w(k) =(1, \ 1) + (\id - \mathfrak C_{W})^{-1}\mathfrak C_{W}\big((1, \ 1)\big)(k).
\]
By  \eqref{6w},
\begin{align}\nn
\|w(k) - (1, \ 1)\|_{L^2(\hat\Sigma)} &\leq \|(\id - \mathfrak C_{W})^{-1}\|_{L^2(\hat\Sigma)\to L^2(\hat\Sigma)} \|\mathfrak C_{-}\|_{L^2(\hat\Sigma)\to L^2(\hat\Sigma)} \|\mathcal W^{err}\|_{L^\infty(\hat\Sigma)}\\
&= O(t^{-\nu}).\label{estmu}
\end{align}
With the help of $w$, the function $m^{er}(k)$ can be represented as 
\[
 m^{er}(k)=(1, \ 1) +\frac{1}{2\pi\I}\int_{\hat\Sigma}\frac{w(z) \mathcal W^{er}(z)dz}{z-k},
\]
and in virtue of \eqref{estmu} and \eqref{nusi1} we obtain as $k\to +\I\infty$,
\[
m^{er}(k) = (1, \ 1) + \frac{1}{2\pi\I } \int_{\hat\Sigma} \frac{(1, \ 1)\mathcal W^{er}(z)}{z-k} dz + E(k),
\]
where
\[
|E(k)|\leq \frac{C}{|k|}\|\mathcal W^{er}\|_{L^2(\hat\Sigma)}\|w (z)- (1, \ 1)\|_{L^2(\hat\Sigma)}\leq \frac{O(t^{-\nu-1})}{|k|},\quad k \rightarrow \infty.
\] 
The term $O(t^{-\nu-1})$ is uniformly bounded with respect to $(x,t)\in\mathcal D$.
In the regime $\re k=0$, $\im k \to +\infty$, we have
\begin{align*}
\frac{1}{2\pi\I } \int_{\hat\Sigma} \frac{(1, \ 1)\mathcal W^{er}(z)}{k-z} dz &=  \frac{f_0(x,t)}{2\I k t^{\nu} }(1, \ -1) + \frac{f_1(x,t)}{2 k^2 t^{\nu}} (1, \ 1)\\
&\quad + O(t^{-\nu})O(k^{-3}) +O(t^{-\nu-1})O(k^{-1}),
\end{align*}
 where $f_{0,1}(x,t)$ are scalar functions uniformly bounded in $\mathcal D$. Furthermore, $O(k^{-s})$ are vector functions depending only on $k$ and $O(t^{-\nu})$, $O(t^{-\nu-1})$ are as above. 
Hence, 
\[
\aligned
& \hat m^{sol}(k,j) = m^{er}(k,j) M(k,j) 
= \mathcal S(k,j) + \frac{f_0(x,t)}{2\I k t^\nu} (1, \ -1)M(k,j) \\
& \quad + \frac{f_1(x,t)}{2 k^2 t^\nu}\mathcal S(k,j) + O(t^{-\nu})O(k^{-3}) +O(t^{-\nu-1})O(k^{-1}), \quad (x,t)\in D_j^{sol};\quad  j=1, \dots,N;
\endaligned
\]
and
\be\label{vne}\aligned
\hat m(k,j)& = m^{er}(k,j)=
 (1,\ 1) + \frac{f_0(x,t)}{2\I k t^\nu} (1, \ -1)
+ \frac{f_1(x,t)}{2 k^2 t^\nu}(1,\ 1) \\
&\quad + O(t^{-\nu})O(k^{-3}) +O(t^{-\nu-1})O(k^{-1}), \quad (x,t)\in D_j;
\quad j=0, \dots, N.
\endaligned\ee
From  \eqref{matrxsol} it follows that
$(1,\ -1)M(k,j)= (1,\ -1) +O(k^2)$. Therefore, \eqref{2} implies that for $(x,t)\in D_j^{sol}$,
\be\label{comper}
m_1(k)m_2(k) = \hat m_1^{sol}(k,j)\hat m_2^{sol}(k,j) = \mathcal S_1(k,j)\mathcal S_2(k,j) + O(t^{-\nu})O(k^{-2}), \quad k\to\infty.
\ee
By \eqref{mathS},
\[\mathcal S_1(k,j)\mathcal S_2(k,j)-1=\frac{2\I\kappa_j\mu_j(x,t) - \mu_j^2(x,t)}{(k-\I\kappa j)(k+\I\kappa_j)}= \frac{-2\kappa_j\gamma_j^2(x,t)}{k^2\Big(1 + (2\kappa_j)^{-1}\gamma_j^2(x,t)\Big)^2}(1 + o(1)).
\]
Comparing this with \eqref{comper} and \eqref{main1} we conclude that in the region $D_j^{sol}$,
\be\label{form} q(x,t)=q_j^{sol}(x,t) + O(t^{-\nu});\qquad q_j^{sol}(x,t):=\frac{-4\kappa_j\gamma_j^2(x,t)}{\left(1 + (2\kappa_j)^{-1}\gamma_j^2(x,t)\right)^2}.\ee
On the other hand,
\[q_j^{sol}(x,t)=-\frac{8\kappa_j^2}{\left(\frac{\sqrt{2\kappa_j}}{\gamma_j(x,t)} +\frac{\gamma_j(x,t)}{\sqrt{2\kappa_j}}\right)^2} = -\frac{8\kappa_j^2}{\left(\E^{\kappa_j x- 4\kappa_j^3 t + \Delta_j} +\E^{-\kappa_j x+4\kappa_j^3 t - \Delta_j}\right)^2},
\]
where
\[\Delta_j=-\frac{1}{2}\log\frac{\gamma_j^2}{2\kappa_j}-\sum_{i=j+1}^N\log\frac{\kappa_j - \kappa_i}{\kappa_i + \kappa_j}.\]
Thus,
\[q_j^{sol}(x,t)=-\frac{2\kappa_j^2}{\cosh^2\Big(\kappa_j x - 4\kappa_j^3 t -\frac{1}{2}\log\frac{\gamma_j^2}{2\kappa_j}-\sum_{i=j+1}^N\log\frac{\kappa_j - \kappa_i}{\kappa_i + \kappa_j}\Big)}.
\]
Note that in $\mathcal D\setminus D_j^{sol}$, this function admits the estimate $O(\E^{-C(\varepsilon) t})$, and
 taking into account the weaker estimate \eqref{form}, we conclude that
\[q(x,t)=\sum_{j=1}^N q_j^{sol}(x,t) + O(t^{-\nu}), \qquad (x,t)\in\bigcup_{j=1}^N D_j^{sol}.\] 
On the other hand, \eqref{vne} implies
\[q(x,t)=O(t^{-\nu}), \qquad (x,t)\in\bigcup_{j=0}^N D_j.\]
This proves Theorem  \ref{mainint}.
\vskip 5mm
{\bf Acknowledgments} I.E. is partially supported by the program "Support of priority research and scientific and technical developments" by  the National Academy of Sciences of Ukraine.

\end{document}